\documentclass{IEEEtran}
\usepackage{cite}
\usepackage{amsmath,amssymb,amsfonts}
\usepackage{algorithmic}
\usepackage{graphicx}
\usepackage{subfig}
\usepackage{textcomp}
\usepackage{lipsum}

\usepackage{color,theorem}
\newtheorem{theorem}{Theorem}
\newtheorem{lemma}{Lemma}

\newtheorem{remark}{Remark}

\newtheorem{proof}{Proof}

\usepackage{tikz-network}
\usepackage{tikz}

\usepackage{amssymb}
\usepackage{amsmath}
\usepackage{mathrsfs}
\usepackage{algorithm}
\usepackage{algorithmic}
\usepackage{bm}
\usepackage{xcolor}
\usepackage{color}
\usepackage{pgfplots}
\usepackage{theorem}
\usepackage{bbm}
\usepackage{balance}

\def\BibTeX{{\rm B\kern-.05em{\sc i\kern-.025em b}\kern-.08em
		T\kern-.1667em\lower.7ex\hbox{E}\kern-.125emX}}
\begin{document}
	\title{A Bi-level Globalization Strategy
		for Non-convex Consensus ADMM and ALADIN }
		\author{Xu Du*\thanks{*Corresponding author},  Jingzhe Wang,  Xiaohua Zhou  and Yijie Mao 
			\thanks{Xu Du, is  with the Institute of Mathematics HNAS, Henan Academy of
				Science, Zhengzhou, China. (e-mail: duxu@hnas.ac.cn).}
			\thanks{Jingzhe Wang is with the Department of Informatics \& Networked
				Systems, School of Computing and Information, University of Pittsburgh,
				Pittsburgh, PA, USA.  (e-mail: jiw148@pitt.edu).}
			\thanks{Xiaohua Zhou and Yijie Mao are with  the School of Information Science
				and Technology, ShanghaiTech University, China.  (e-mail: \{zhouxh3,maoyj\}@shanghaitech.edu.cn).}
		}
		
		\maketitle
		
		\begin{abstract}
			
			In this paper, we formally analyze global convergence in the realm of distributed consensus optimization. Current solutions have explored such analysis, particularly focusing on consensus alternating direction method of multipliers (C-ADMM), including convex and non-convex cases. While such efforts on non-convexity offer elegant theory guaranteeing global convergence, they entail strong assumptions and complicated proof techniques that are increasingly pose challenges when adopted to real-world applications. To resolve such tension, we propose a novel bi-level globalization strategy that not only guarantees global convergence but also provides succinct proofs, all while requiring mild assumptions. We begin by adopting such a strategy to perform global convergence analysis for the non-convex cases in C-ADMM. Then, we employ our proposed strategy in consensus augmented Lagrangian based alternating direction inexact Newton method (C-ALADIN), a more recent and generalization of C-ADMM. Surprisingly, our analysis shows that C-ALADIN globally converges to local optimizer, complementary to the prior work on C-ALADIN, which had primarily focused on analyzing local convergence for non-convex cases.


			

		\end{abstract}
		
		\section{Introduction}

		Alternating direction method of multipliers (ADMM) is the most well-known algorithm within the realm of distributed optimization.
		The primary idea of ADMM is to decompose a complex optimization problem into simpler sub-problems, each can be solved in parallel by different agents.  
		Since its first proposal in \cite{glowinski1975approximation,Gabay1976}, there has been much effort to develop ADMM. {\color{black}{ Such efforts have mainly focused on the following two lines of research.}}
		The first line of work, including but not limited to \cite{wen2010,boyd2011distributed}, has explored potential applications that could harness the advantages of ADMM, such as machine learning \cite{zhou2022federated}, signal processing \cite{Boyd2011}, wireless communication \cite{xu2021rate}. To enrich the theoretical ground of ADMM, the second line of work, represented by \cite{hong2016convergence,he20121,Shi2014,ling2015dlm}, has turned its attention to revisiting the structure of ADMM. In particular, the most influential work is that of He and Yuan  \cite{he20121}, which first analyzed the convergence rate of ADMM for general convex cases and investigated the sub-linear convergence result. Later, Shi et al. extended this trajectory by introducing the assumption of strongly convex \cite{Shi2014}, leading to a linear convergence rate of consensus ADMM (C-ADMM). The aforementioned advancements, however, have been exclusively  to deal with convex problems. The application of for non-convex consensus problems remains limited.
		
		
		
		In tackling the convergence challenge \cite{Houska2016, Hong2016} caused by non-convexity, several authors, including Hong and others
		\cite{hong2016convergence,li2015global,wang2018convergence,yang2022proximal}, have suggested enforcing global convergence of ADMM with sufficiently large parameter of the augmented term. As a special case, Wang  et al. proposed a variant of ADMM, known as Bregman ADMM \cite{wang2018convergence} (an extension of proximal point method).
		Later, Wang  et al. \cite{wang2019global} explored  some non-convex non-smooth special cases that can be effectively managed by ADMM. All the aforementioned methods are special instances of the \emph{proximal ADMM} algorithm proposed in   \cite{yang2022proximal}. We refer \cite{yang2022survey} as a comprehensive
		summary of ADMM variations to interested readers. We are aware of that most of the existing works of non-convex ADMM rely on the following two different methods: a) showing monotonic decreasing of a critical Lyapunov function \cite{wang2019global,yang2022proximal}, b) showing  a corresponding (augmented) Lagrangian function \cite{hong2016convergence} converging to a limit point. Though applicable to C-ADMM,  these two approaches inevitably entail quite complicated proofs. Specifically, the former method requires positive definiteness of the Lyapunov function, constructed hard in nature, while the latter requires the smoothness of the objective function, thereby narrowing the scope of applicability of C-ADMM.


		
		Fortunately, in pursing concise proofs, Houska et al. proposed the notion of T-ALADIN (typical augmented Lagrangian based alternating direction inexact Newton method) \cite{Houska2016}, a solution that provides theoretical local convergence guarantees when tackling general distributed non-convex optimization problems, assuming
		a good initial point is available. It was not until seven years later that Du and Wang realized that T-ALADIN might have a more efficient form of implementation  \cite{Du2023} (C-ALADIN, consensus ALADIN) for distributed consensus problem in the ALADIN framework. While the approaches showcased promising results, their guarantees were limited to local convergence.
		
		{\color{black}{		In this paper, depart from the conventional  \emph{Gauss-Seidel decomposition}, 
				instead, we delve into the  parallelizable implementation of C-ADMM and C-ALADIN, based on \emph{Jacobian decomposition}.
				Inspired by the globalization approach proposed in \cite[Section 6]{Houska2016},  we initiate the study of bi-level globalization, which ensures that
				C-ADMM and C-ALADIN globally converge to local optimizer. 
				In summary, this paper has the following key contributions:

				a) A novel \emph{bi-level\footnote{In fact, there are several papers also works on  bi-level  optimization structure, see \cite{sun2023two,Engelmann2020, stomberg2022decentralized,wang2022confederated,Houska2013}, however, that have different purpose compare with this paper.} proximal term}  is proposed to handle the non-convexity in the classical distributed consensus optimization optimization problem. Two distributed optimization algorithms, named CADMM-Prox and CALADIN-Prox are then proposed for combining bi-level globalization with C-ADMM and C-ALADIN, respectively. To our best knowledge, this is the first work that applies bi-level globalization to C-ADMM and C-ALADIN.  It's worth noting the proposed bi-level globalization approach is promising as it ensures the strong convexity of the lower level problem and is guaranteed to converge to the (local) optimal solution of the original non-convex problem. This contrasts with the globalization employed in \cite[Section 6]{Houska2016}, which is used to tuning the step size only and only local optimal solution is guaranteed.


				b) Surprisingly, we obtain a simple and straightforward proof procedure thanks to the use of globalization.

		}}

		The rest of this paper is organized as follows. 
		In Section \ref{sec: Merit and Prox}, we review the preliminaries regarding globalization strategy.  
		In Section \ref{sec: Conservative}, with the  $L_2$-norm proximal term and the $L_1$-norm merit function, we propose our globalization strategy for C-ALADIN and C-ADMM.
		In Section \ref{sec: Convergence}, we establish the global convergence theory.
		In the end, Section \ref{sec: conclusion} concludes this paper.

		\section{Preliminaries}\label{sec: Merit and Prox}
		
		A distributed consensus optimization problem\footnote{In the rest of this paper, we assume that, for all local minimizers of Problem \eqref{eq: DOPT_C}, the second order sufficient condition (SOSC) and linear independence constraint qualification (LICQ) are satisfied.} can be summarized as follows
		\begin{equation}\label{eq: DOPT_C}
			\begin{split}
				\min_{x,y}\;\;&\mathop{\sum}_{i=1}^{N}  f_i(x_i)\\
				\mathrm{s.t.}\;\;\;\;&\; x_i=y \;| \;\lambda_{i},\quad \forall i\in \{1,2, \dots,N\},
			\end{split}
		\end{equation}
		where $f_i$s are potentially non-smooth, non-convex, and bounded from below. {\color{black}{Here, $x_i\in \mathbb{R}^n$ denotes the local variable of agent $i$, $x\in \mathbb{R}^{N\cdot n}$ collects $x_i$s as a vector, $y\in \mathbb{R}^n$ represents the global variable, and $\lambda_i$s indicate the corresponding Lagrangian multiplier of agent $i$. 
		}}
		\subsection{An overview of C-ADMM and C-ALADIN}
		The augmented Lagrangian function of  
		Problem \eqref{eq: DOPT_C} can be formalized as follows
		\begin{equation}\label{eq: augmented Lagrangian}
			\begin{split}
				\mathscr{L}_\rho (x,y,\lambda) \overset{\text{def}}{=}&\; \mathop{\sum}_{i=1}^{N}  f_i(x_i) \\
				&+  \mathop{\sum}_{i=1}^{N} \lambda_{i}^\top \left( x_i-y \right)+\frac{\rho}{2}\mathop{\sum}_{i=1}^{N} \|x_i-y\|^2.
			\end{split}
		\end{equation}
		Based on  \eqref{eq: augmented Lagrangian}, C-ADMM  can be then represented as \eqref{eq: ADMM} 
		\begin{equation}\label{eq: ADMM}
			\left\{
			\begin{split}
				x_i^+ =&\; \mathop{\arg\min}_{x_i} f_i(x_i) +\lambda_{i}^\top \left( x_i-y \right) + \frac{\rho}{2} \|x_i-y\|^2;\\
				y^+ = &\; \frac{1}{N} \mathop{\sum}_{i=1}^{N}\left(x_i^+ + \frac{1}{\rho} \lambda_{i} \right); \\
				\lambda_{i}^+ =&\; \lambda_{i}+ \rho\left(x_i^+-y^+\right).\\
			\end{split}\right.
		\end{equation}
		If $f_i$s are convex, \eqref{eq: ADMM} is guaranteed to converge globally within finite iterations \cite{boyd2011distributed}.
		
		By combining C-ADMM with sequential quadratic programming (SQP) method\footnote{
			The main difference between \eqref{eq: ADMM} and \eqref{eq: ALADIN} is that C-ALADIN updates the dual $\lambda$ together with the global variable $y$, while C-ADMM updates the all the variables separately.}, we get C-ALADIN  \cite{Du2023}, as per \eqref{eq: ALADIN}. It only guarantees local convergence for non-convex problems when starting from a good initial guess.
		\begin{equation}\label{eq: ALADIN}
			\left\{
			\begin{split}
				&x_i^+ = \mathop{\arg\min}_{x_i} f_i(x_i) +\lambda_{i}^\top \left( x_i-y \right) + \frac{\rho}{2} \|x_i-y\|^2;\\
				&\left(y^+, \lambda^+ \right) \\
				=&\left\{
				\begin{split}
					\vspace{-1cm}	&\mathop{\mathrm{\arg}\mathrm{\min}}_{\Delta \bar{x}, y}\;\mathop{\sum}_{i=1}^{N}  \left(\frac{1}{2}\Delta\bar{ x}_i^\top B_i\Delta\bar{ x}_i+\partial f_i(x_i^+)^\top \Delta\bar{ x}_i \right)\\
					&\quad\mathrm{s.t.} \qquad\Delta\bar{ x}_i+x_i^+=y\; |\lambda_{i}\\
				\end{split}\right\}.\\
			\end{split}\right.
		\end{equation}
		Here $B_i\succ 0$ and $\partial f_i$ denote the local Hessian approximation and (sub)gradient of $f_i$, respectively. 
		
		\begin{remark}
			It is worth noting that, in solving large-scale optimization problems, C-ADMM is a \emph{first-order} algorithm that is computationally tractable. However, it lacks theoretical guarantees for non-convex problems. Conversely, C-ALADIN does provide rigorous convergence analysis, but it only has local convergence guarantees instead of global ones. Our proposed globalization strategy (detailed in \ref{sec: Convergence}) solved the two challenges.
			
		\end{remark}
		
		\subsection{The $L_2$-norm Proximal Term  $\&$ the $L_1$-norm Merit Function}
		Since $f_i$s are in general non-convex, we introduce the $L_2$-norm proximal term \cite{lemaire1989proximal,iusem1999augmented} as follows
		\begin{equation}\label{eq: gamma} 
			\frac{\gamma}{2} \|x_i  - z \|^2.
		\end{equation}
		By incorporating \eqref{eq: gamma} with an upper global variable $z$ and a sufficiently large parameter $\gamma>0$, we guarantee the strong convexity of $F_i^{{z}}$ with respect to $x_i$, formulated in Equation \eqref{eq: Proximal Term}
		\begin{equation}\label{eq: Proximal Term}
			F_i^{ z}(x_i)  \overset{\text{def}}{=}  f_i(x_i) + \frac{\gamma}{2} \|x_i  - z \|^2.
		\end{equation}
		
		{\color{black}{	Instead of solving \eqref{eq: DOPT_C} directly, in the following sections, two optimization algorithms are proposed for \eqref{eq: DOPT_C2} based on bi-level optimization. Specifically, 
				the update of $x$ and $y$ will be performed in the \emph{lower level} while the update of $z$  will be done in the \emph{upper level}.
				In the lower level, with a fixed $z$, we focus on searching the optimizer of \eqref{eq: DOPT_C2}.
				\begin{equation}\label{eq: DOPT_C2}
					\begin{split}
						\min_{x,y}\;\;&\mathop{\sum}_{i=1}^{N}  F^{ z}_i(x_i)\\
						\mathrm{s.t.}\;\;&\;\; x_i=y \;| \;\lambda_{i},\quad \forall i\in \{1,2, \dots,N\}.
					\end{split}
				\end{equation}
				Note that, $\gamma$ guarantees the strong convexity of the \emph{lower level problem} \eqref{eq: DOPT_C2}, making it easier to be solved. 
				Then, the update of $z$ -- a key ingredient within the \emph{upper level} -- enforces the solution of Problem~\eqref{eq: DOPT_C2} asymptotically reaching towards the (local) optimal solution of Problem \eqref{eq: DOPT_C} eventually.} }


		With the lower global variable $y$, the $L_1$-norm \emph{merit function} \cite[Chapter 18.3]{Nocedal2006} of the sub-problem can be then expressed as Equation \eqref{eq: merit}
		\begin{equation}\label{eq: merit}
			\Phi_i^{(z, y)}(x_i) \overset{\text{def}}{=} F^{z}_i(x_i)+ \sigma_i  \| x_i -y  \|_1,
		\end{equation}
		that aims to measure the performance of the objective and the equality constraints violation of each sub-problem. 	Here $\sigma_i$ is sufficiently large, such that, $\sigma_i \geq \|\lambda_i\|_\infty \; \forall i$ \cite{Nocedal2006}. 
		Obviously, with Equation \eqref{eq: merit}, the $L_1$-norm merit function of problem \eqref{eq: DOPT_C2} 
		can be described as \eqref{eq: Merit}
		\begin{equation}\label{eq: Merit}
			\Phi^{(z, y)}(x) \overset{\text{def}}{=} \; \sum_{i=1}^{N} F^{ z}_i(x_i)+ \sum_{i=1}^{N} \sigma_i \| x_i -y  \|_1.
		\end{equation}
		In the following sections, \emph{merit function} only relates to Equation \eqref{eq: Merit}.
		Note that, 
		\begin{equation}\label{eq: equal}
			\sum_{i=1}^{N} f_i(z)=\Phi^{(z, z)}(z)
		\end{equation}
		is an important property that helps with establishing the global convergence theory through this paper.
		The analysis in the following sections depends on the update of the triple $(x, y, z)$.
		Formally, 
		we define 
		\begin{equation}\label{eq: increment}
			\left\{
			\begin{split}
				\Delta x_i &\overset{\text{def}}{=} x_i^+ - y;\\
				\Delta\tilde{ x}_i&\overset{\text{def}}{=} y^+-x_i^+.
			\end{split}\right.
		\end{equation}
		Here $(\cdot)^+$ denotes the update of corresponding decision variables while $(\cdot)^-$  represents such variables from the latest iteration.

		\section{Proposed Optimization Frameworks}\label{sec: Conservative}
		In this section, we take advantage of our proposed globalization strategy and propose two novel algorithms to solve Equation~\eqref{eq: DOPT_C}, namely CADMM-Prox and CALADIN-Prox.
		Specifically,  in Subsection \eqref{eq: Conservative Operation}, we first propose the updating strategy of the local variable $x_i$ and global variable $y$. Then, we detail the entire optimization frameworks in
		Subsection \eqref{sec: ALADIN}.
		
		\subsection{Conservative Update of the Lower Level}\label{eq: Conservative Operation}
		
		The lower level of C-ALADIN and C-ADMM has two main steps: a) local nonlinear programming (NLP) minimization of $x_i$ (Section \ref{sec: local NLP}), b) consensus step ~(Section \ref{sec: global update}). 
		\subsubsection{Local Primal Variable Update}\label{sec: local NLP}
		
		With the proximal term \eqref{eq: gamma}, we suggest that updating the local primal variable $x_i$ through approximated reformulation. Concretely, 
		for C-ALADIN, 
		we linearize the upper objective $F^z_i(x_i)$, such that
		\begin{equation}\label{eq: local update}
			\begin{split}
				{x_i}^+=&\mathop{\arg\min}_{x_i} F^{ z}_i(x_i)+\lambda_i^\top  x_i+\frac{1}{2}\|x_i-y\|^2_{B_i}\\
				\approx& \mathop{\arg\min}_{x_i} \left(   \partial f_i(y) + \gamma (y-z)\right)^\top(x_i-y)  \\
				&   +\lambda_i^\top  x_i+\frac{1}{2}\|x_i-y\|^2_{B_i}\\
				=
				&	 y+ B_i^{-1}\left(\gamma (z-y) - \lambda_i -\partial f_i(y)  \right)
			\end{split}
		\end{equation}
		with $B_i\succ 0$. Here $\|\cdot\|_{B_i}$ denotes the Mahalanobis distance \cite{wang2018convergence} with symmetric positive definite matrix $B_i$.
		

		For C-ADMM,
		by setting $B_i = \rho I$, Equation \eqref{eq: local update} degrades into
		\eqref{eq: local update3}
		\begin{equation}\label{eq: local update3}
			\begin{split}
				{x_i}^+=	y+	\frac{1}{\rho}\left(\gamma (z-y) - \lambda_i -\partial f_i(y)  \right).
			\end{split}
		\end{equation}

		Indeed, alternative approaches to updating the local variables may also work. We detail this part in Appendix 
		\ref{app: cmt}.

		
		
		\subsubsection{Lower Level Global Variable Update (Consensus Step)}\label{sec: global update}
		
		In the C-ALADIN framework, if we have the local Hessian approximation matrices $B_i\approx \nabla^2F^z(x_i^+)\succ 0$, 
		the previous  lower global variable $y^-$, $\beta>0$, and the local (sub)gradient $g_i$ formalized as follows
		\begin{equation}\label{eq: BFGS}
			\begin{split}
				g_i&\overset{\text{def}}{=} \nabla F^{ z}(x_i^+)\\
				&=\partial f_i(x_i^+) +\gamma ( x_i^+- z)\\
				&\approx B_i(y-x_i^+)-\lambda_i,
			\end{split}
		\end{equation}
		we can then derive the update of the lower global variable $y^{\text{ALADIN}}$ in Equation \eqref{eq: consensus QP1}
		\begin{equation}\label{eq: consensus QP1}
			\begin{aligned}
				&\left(y^{\text{ALADIN}}\right)^+\\
				=&\left\{
				\begin{split}
					&\mathop{\mathrm{\arg}\mathrm{\min}}_{\Delta \bar{ x}, y}\quad\mathop{\sum}_{i=1}^{N}  \left(\frac{1}{2}\Delta\bar{ x}_i^\top B_i\Delta\bar{ x}_i+g_{i}^\top \Delta\bar{ x}_i \right)\\
					&\qquad\qquad+\frac{\beta}{2} \|y-y^-\|^2\\
					&\quad\mathrm{s.t.} \qquad\Delta\bar{ x}_i+x_i^+=y\; |\;\lambda_{i}\\
				\end{split}\right\}.
			\end{aligned}
		\end{equation}
	Note that, Equation \eqref{eq: consensus QP1} has a close form as follows
	\begin{equation}\label{eq: y_ALADIN}
		\begin{split}
			&	\left(y^{\text{ALADIN}}\right)^+\\	=& \left(\sum_{i=1}^{N} B_i+\beta I \right)^{-1}\left(\beta y^- +\sum_{i=1}^{N}\left( B_ix_i^+- g_i\right) \right).
		\end{split}
	\end{equation}
	
	By setting $B_i=\rho I$ and $\lambda_i= -g_i$, Equation \eqref{eq: y_ALADIN} boils down to \eqref{eq: ADMM global} that represents the lower global variable update of C-ADMM.
	\begin{equation}\label{eq: ADMM global}
		\begin{split}
			\left(y^{\text{ADMM}}\right)^+
			=\frac{1}{N\rho +\beta}\left( \beta y^-  + \sum_{i=1}^{N}\left( \rho x_i^+ + \lambda_i\right) \right).
		\end{split}
	\end{equation}
	
	\subsection{Proposed Algorithms}\label{sec: ALADIN}
	\subsubsection{CALADIN-Prox}
	With the approximate update of local variable $x_i$ (Equation \eqref{eq: local update}) and global variable $y$ (Equation \eqref{eq: consensus QP1}), the linearized form of C-ALADIN, equipped with globalization strategy, can be then expressed as Algorithm
	\ref{alg: Globalization of ALADIN}.

	\begin{algorithm}[t]
		\small
		\caption{CALADIN-Prox}
		\textbf{Initialization:} Initial guess $\lambda_i$, choose upper global variable $z$, set lower global variable $y = z$, local Hessian approximation $B_i\succ 0$. Set $\gamma>0,  \rho>0, \beta\geq0, \sigma_i = \|\lambda_i\|_\infty$. \\
		\textbf{Repeat:}
		\begin{enumerate}
			\item  Inexactly updates the local variable $x_i$
			with 
			Equation \eqref{eq: local update} or \eqref{eq: multi1} 
			where $\lambda_{i}=B_i(x_i^--y)-g_i$.
				
				\item Evaluate the new sub-gradient 
				\begin{equation*}\label{eq: subgradient}
					\begin{split}
						g_i=B_i(y-x_i^+)-\lambda_i.
					\end{split}
				\end{equation*} from each ${x_i}^+$ and upload to the master.

				\item  Update the lower global variable $y$:
				\begin{equation*}\label{eq: BFGS QP}
					\begin{split}
						y^+=& \left(\sum_{i=1}^{N} B_i+\beta I \right)^{-1}\left(\beta y +\sum_{i=1}^{N}\left( B_ix_i^+- g_i\right) \right).
					\end{split}
				\end{equation*}
				
				\item Set $\sigma_i^+ = \|B_i(x_i^+-y^+)-g_i\|_\infty$ if 
				\[\sigma_i<\|B_i(x_i^+-y^+)-g_i\|_\infty. \]  
				
				\item	Update the upper global variable $z^+ = y^+$  if \[\Phi^{(z, y^+)}(y^+) <\Phi^{(z, y)}(y).\]

				\item Download the updated variables:
				\begin{equation*}
					\left\{
					\begin{array}{l}
						\begin{split}
							y&\leftarrow y^+;\\
							0 &\;\text{or}\; 1\qquad \text{1 for updated $z$ and 0 for no reaction \footnotemark{}}.
						\end{split}
					\end{array}
					\right.
				\end{equation*} 
			\end{enumerate}
			\label{alg: Globalization of ALADIN}
		\end{algorithm}
		\footnotetext{Optionally update the  local Hessian approximation $B_i\approx \nabla^2 F_i^{z^+}(z^+)$ of each sub-agent.}

		\begin{remark}
			The numerical value of $\left(\sum_{i=1}^{N} B_i +\beta I\right)^{-1}$ can be prepared for acceleration. As discussed in \cite{Du2023}, when setting $B_i = \rho I$, C-ALADIN is still different from C-ADMM in updating the dual variables.  
		\end{remark}

		\subsubsection{CADMM-Prox}
		Similar to Algorithm \ref{alg: Globalization of ALADIN},  C-ADMM with globalization strategy can be summarized as Algorithm \ref{alg: Globalization of ADMM}, supported by Equation \eqref{eq: local update3}  and \eqref{eq: ADMM global}.
		
		\begin{remark}
			A proximal term has also been introduced in \emph{Proximal ADMM} \cite{yang2022proximal} as $\frac{\gamma}{2}\|x_i  - y\|^2$, however, only a single-level optimization structure is considered in that work.  
			The advantage of the bi-level structure
			in Algorithm \ref{alg: Globalization of ADMM} is that the upper level global variable $z$ in $\frac{\gamma}{2} \|x_i  - z \|^2$ is not updated per iteration in the lower level, which gives the strong convexity of \eqref{eq: DOPT_C2}. In fact, for single-level ADMM, the strong convexity of the augmented sub-problem in Equation \eqref{eq: ADMM} is not a sufficient condition for ADMM to converge. See a contrary example \cite[Example 2.1]{Houska2016}. 
		\end{remark}

		\section{Convergence Analysis}\label{sec: Convergence}
		In this section, we establish the global convergence theory of Algorithm \ref{alg: Globalization of ALADIN} and \ref{alg: Globalization of ADMM}, which is the major contribution of this paper.
		In Subsection \eqref{sec: QP}, we begin with analyzing the convergence behavior of the lower level of Algorithms \ref{alg: Globalization of ALADIN} and \ref{alg: Globalization of ADMM}. Later, in Subsection \eqref{sec: Global Convergence Theorem}, we establish the global convergence theory with our proposed globalization strategy in the sense of upper level. The combination of the two finally yields the global convergence.
		\begin{algorithm}[t]
			\small
			\caption{CADMM-Prox}
			\textbf{Initialization:}Initial the upper and lower global parameter $z, y$ and the dual $\lambda_i$. Set $\gamma>0, \rho>0, \beta\geq0, \sigma_i = \|\lambda_i\|_\infty$.  \\
			\textbf{Repeat:}
			\begin{enumerate}
				\item 
				Evaluate $\lambda_i^+=\lambda_i+\rho(x_i-y)$, then
				inexactly update the local variable as $x_i^+$ with Equation \eqref{eq: local update3} or \eqref{eq: multi2} on the agent side and upload to the master.

				\item  Update the lower global variable $y$:
				\begin{equation*}
					\begin{split}
						y^+ &= \frac{1}{N\rho +\beta}\left( \beta y  + \sum_{i=1}^{N}\left( \rho x_i^+ + \lambda_i\right) \right).
					\end{split}
				\end{equation*}
				
				\item Set $\sigma_i^+ = \|\lambda_i+\rho(x_i^+-y^+)\|_\infty$ if 
				\[\sigma_i<\|\lambda_i+\rho(x_i^+-y^+)\|_\infty. \]  
				
				\item Update the upper global variable $z = y^+$ if 
				\[\Phi^{(z, y^+)}(y^+) < \Phi^{(z, y)}(y).\]

				\item Broadcast the updated variables:
				\begin{equation*}
					\left\{
					\begin{array}{l}
						\begin{split}
							y&\leftarrow y^+\\
							0 &\;\text{or}\; 1\qquad \text{1 for updated $z$ and 0 for no reaction.}
						\end{split}
					\end{array}
					\right.
				\end{equation*} 
			\end{enumerate}
			\label{alg: Globalization of ADMM}
		\end{algorithm}




		\subsection{Lower Level Merit Function Decreasing Guarantees}\label{sec: QP}
		Given a nonlinear mapping $\Lambda(\xi): \mathbb{R}^n\rightarrow \mathbb{R}$, 
		and the increment of the decision variable $p\in \mathbb{R}^n$,  
		we have the concept of \emph{directional derivative}, defined as \eqref{eq: directional derivative1}
		\begin{equation}\label{eq: directional derivative1}
			\text{D} \Lambda(\xi)[p] \overset{\text{def}}{=} \lim_{t\rightarrow 0^+} \frac{\Lambda(\xi+tp)-\Lambda(\xi)}{t}. 
		\end{equation}
		Before providing the decent condition of the merit function \eqref{eq: Merit},
		we first give the following lemma.
		
		\begin{lemma}\label{Lemma: directional derivative}
			With the definition of $\text{D} \Lambda(\xi)[p]$ in Equation \eqref{eq: directional derivative1}, assume that $\Delta\tilde x$ is generated by the QP iteration from \eqref{eq: consensus QP1},	then the directional derivative of Equation \eqref{eq: Merit} is represented as
			\begin{equation}\label{eq: directional derivative0}
				\begin{split}
					&\text{D} \Phi^{( z,y^+ )}(x^+)[\Delta\tilde x]\\
					=& \sum_{i=1}^{N}   \nabla	F_i^{ z}(x_i^+)^\top \Delta\tilde x_i -	\sum_{i=1}^{N}\sigma_i \| \Delta\tilde x_i  \|_1	+O(t).
				\end{split}
			\end{equation} 
			Here, $y^+$ can be pre-computed by following Equation
			\eqref{eq: y_ALADIN} in Algorithm \ref{alg: Globalization of ALADIN} or \eqref{eq: ADMM global} in Algorithm \ref{alg: Globalization of ADMM}.
			
		\end{lemma}
		
		\begin{proof}
			See Appendix \ref{app: Proof of Lemma1}. 	\hfill$\blacksquare$
		\end{proof}
		Moreover, for the $L_1$ norm gap $\|x-y\|_1$, the directional derivative with respect to $x$ at point $y$, in direction $p$, can be denoted as follows 
		\begin{equation}\label{eq: directional derivative2}
			\text{D} \|x-y\|_1(y)[p] = \|p\|_1.
		\end{equation}

		With Equation \eqref{eq: directional derivative0} and \eqref{eq: directional derivative2}, now we can show that the  merit function \eqref{eq: Merit} is decreasing in the lower level of Algorithm  \ref{alg: Globalization of ALADIN} and \ref{alg: Globalization of ADMM} , detailed in Theorem \ref{theorem: decrease1} and \ref{theorem: decrease}. Note that, the following analysis mainly relates to the increment with respect to the local variables $x_i$s introduced in Equation \eqref{eq: increment}.
		\subsubsection{Merit Function Analysis at the Local NLP Update Step }
		The following theorem guarantees the decrement of the merit function with the local primal $x_i$ update for CALADIN-Prox. 
		\begin{theorem}\label{theorem: decrease1} With the local update \eqref{eq: local update},
			the directional derivative of the lower merit function \eqref{eq: Merit}  with respect to the local NLP \eqref{eq: local update}  has the following property:
			\begin{equation}\label{eq: purpose}
				\text{D} \Phi^{( z, y)}(y)[\Delta x]< 0
			\end{equation}
			will hold if 
			\begin{equation}\label{eq: assumption}
				\begin{split}
					\sum_{i=1}^{N} \lambda_i^\top \Delta x_i +\sum_{i=1}^{N} \|\Delta x_i\|^2_{B_i}>\sum_{i=1}^{N}\sigma_i  \| \Delta x_i  \|_1.
				\end{split}
			\end{equation}
		\end{theorem}
		\begin{proof}
			See Appendix \ref{app: theorem1}.	\hfill$\blacksquare$
		\end{proof}
		Since $x_i^+= y + \Delta x_i$,
		Theorem \ref{theorem: decrease1} implies
		\begin{equation}
			\Phi^{( z, y)}(x_i^+) <\Phi^{( z, y)}(y).
		\end{equation}
		By setting $B_i = \rho I$ and updating with Equation \eqref{eq: local update3}, the above statement also works for CADMM-Prox.
		
		Note that, the decreasing analysis of the merit function on local NLP update, described by Equation \eqref{eq: local update0}, \eqref{eq: multi1},  \eqref{eq: local update2} and \eqref{eq: multi2}, are similar to Theorem \ref{theorem: decrease1}. We thus do not show detail here.
		
		
		\subsubsection{Merit Function Analysis at the Lower Level Global Variable Update Step}
		The following theorem  provides the descent condition of the merit function \eqref{eq: Merit}. Such a condition can evaluate the results of the consensus step in Algorithm \ref{alg: Globalization of ALADIN}.
		

		\begin{theorem}\label{theorem: decrease}
			With Lemma \ref{Lemma: directional derivative},
			the directional derivative  of the lower merit function \eqref{eq: Merit}
			with respect to the consensus QP \eqref{eq: consensus QP1} has the following property:
			\begin{equation}\label{eq: consensus descent}
				\begin{split}
					&\text{D} \Phi^{( z, y^+)}(x^+)[{\Delta\tilde x}]< 0,\\
				\end{split}
			\end{equation}
			if \begin{equation}\label{eq: condition}
				\sigma_i>\|\lambda_i^+\|_\infty.
			\end{equation}
		\end{theorem}
		\begin{proof}
			See Appendix \ref{app: Theorem2}.	\hfill$\blacksquare$
		\end{proof}
		After updating the lower global variable $y$, the linear consensus coupling of Equation \eqref{eq: consensus QP1} holds, such that $x_i^++\Delta\tilde{  x_i} = y^+$.
		Therefore,  
		\begin{equation}\label{eq: global descent}
			\Phi^{( z, y^+)}(y^+) < \Phi^{( z, y)}(x^+)
		\end{equation}
		is guaranteed with the Theorem \ref{theorem: decrease}. 
		The above theorem shows that Equation \eqref{eq: consensus QP1} will give a decent direction of the lower merit function.
		Similar to the above discussion, 
		by replacing $B_i$ with $\rho I$ and applying Appendix \ref{app: Theorem2} again,  \eqref{eq: global descent} also works for CADMM-Prox.   

		By combining Theorem \ref{theorem: decrease1} and \ref{theorem: decrease}, the merit function is guaranteed to decrease with respect to the lower level of Algorithm \ref{alg: Globalization of ALADIN} and \ref{alg: Globalization of ADMM} in only one step, such that
		\begin{equation}\label{eq: lower}
			\Phi^{( z,y^+)}(y^+)<  \Phi^{( z,  z)}( z).
		\end{equation}
		
		\begin{remark}\label{remark: lower}
			If one of \eqref{eq: assumption} and \eqref{eq: condition} fail to meet, after one step of the lower level,  the update of the upper global variable $z$ of Algorithm \ref{alg: Globalization of ALADIN} and \ref{alg: Globalization of ADMM} can not be guaranteed. Fortunately, with the global convergence theory of convex C-ADMM \cite{Shi2014}, Equation \eqref{eq: lower} still holds within finite steps for the lower level of Algorithm \ref{alg: Globalization of ADMM}. \end{remark}

		For convex C-ALADIN consisting of a special case such that $B_i=\rho I$, \cite{Du2023} shows the linear convergence of strongly convex problems. The following statement will show that C-ALADIN can also globally converge within finite iteration for convex problems in a more general setting. The strongly convex distributed scenario relates to the lower level of Algorithm \ref{alg: Globalization of ALADIN} such that \eqref{eq: lower} also holds.

		
		\begin{theorem}\label{the: lya}
			Assume $(y^*, \lambda^*)$ denotes the KKT (Karush-Kuhn-Tucker) point   of Problem \eqref{eq: DOPT_C2},
			for general local Hessian approximation $B_i\succ 0$,
			the Lyapunov function \eqref{eq: lya}
			\begin{equation}\label{eq: lya}
				\mathcal L(y,\lambda) =  \sum_{i=1}^{N}      \|y-y^*\|_{B_i}^2+  \sum_{i=1}^{N}\|\lambda_{i} - \lambda_{i}^* \|_{B_i^{-1}}^2
			\end{equation} 
			is monotonely decreasing with the lower level of   Algorithm \ref{alg: Globalization of ALADIN}, such that
			\begin{equation*}\label{eq: decrease}
				\mathcal L(y^+,\lambda^+)-\mathcal L(y,\lambda) \leq 0.
			\end{equation*} 
		\end{theorem}
		\begin{proof}
			See Appendix \ref{app: lya}.	\hfill$\blacksquare$
		\end{proof}
		Since Equation \eqref{eq: lya} is positive definite, it converges to $0$ in finite iterations by adopting C-ALADIN.
		For strongly convex problems,  we can further show the linear rate of  C-ALADIN by applying \cite[Theorem 7]{Du2023} again.

		\subsection{Enforcing Global Convergence to Local Optimizer}\label{sec: Global Convergence Theorem}
		
			%
			%
		
		With the discussion of Equation \eqref{eq: lower} in Subsection \eqref{sec: QP},
		now the global convergence theory of Algorithm \ref{alg: Globalization of ALADIN} and \ref{alg: Globalization of ADMM} can be established with the following theorem.
		
		\begin{theorem}\label{theorem: global convergence}
			With the lower level of Algorithm \ref{alg: Globalization of ALADIN} and \ref{alg: Globalization of ADMM}, the decrement of $\Phi^{(\cdot,\cdot)}(\cdot)$ always holds, such that 
			\begin{equation}\label{eq: global decent}
				\Phi^{( z,  z)}( z) -\Phi^{(y^+,y^+)}(y^+) >   \frac{\gamma N}{2} \|  y^+- z \|^2	>0.
			\end{equation}
			Here, $y^+$ is a strict global minimizer of Problem \eqref{eq: DOPT_C2} generated by the lower level of  Algorithm \ref{alg: Globalization of ALADIN}  or \ref{alg: Globalization of ADMM} with ${z}$ ($y^+\neq {z}$) as an initial point.
		\end{theorem} 
		\begin{proof}
			See Appendix \ref{APP: global convergence}.	\hfill$\blacksquare$
		\end{proof}


		By setting $z^+=y^+$, Equation \eqref{eq: global decent} implies
		\begin{equation}\label{eq: merit decrease}
			\Phi^{(z^k, z^k)}(z^k)<\Phi^{(z^0, z^0)}(z^0) - \frac{\gamma N}{2} \sum_{j=0}^{k-1}\|z^j - z^{j+1} \|^2,
		\end{equation}
		where $k$ is the iteration index of the upper level.
		Since $\sum_{i=1}^N f_i(x_i)$ is bounded below, 
		with the proposed globalization routine,  the following limit always exists
		\[\Phi^{(z^*,z^*)}(z^*) = \lim_{k\rightarrow \infty} \Phi^{(z^k, z^k)}(z^k)>-\infty.\]
		Here $z^*$ donates the limit point of $\Phi^{(\cdot,\cdot)}(\cdot)$.
		By taking the limit of Equation \eqref{eq: merit decrease} on both sides, we have 
		\[\Phi^{(z^*,z^*)}(z^*) < \Phi^{(z^0, z^0)}(z^0) - \frac{\gamma N}{2} \sum_{j=0}^{\infty}\|z^j - z^{j+1} \|^2  \]
		which implies 
		\begin{equation}\label{eq: infity}
			\begin{split}
				&\sum_{j=0}^{\infty}\|z^j - z^{j+1} \|^2 \\
				<& \frac{2}{\gamma N}\left( \Phi^{(z^0, z^0)}(z^0)-\Phi^{(z^*,z^*)}(z^*)\right).
			\end{split}
		\end{equation}
		Notice that,
		the left hand side of Inequality \eqref{eq: infity} must be bounded such that 
		\[ \lim_{j\rightarrow \infty}\|z^j - z^{j+1} \|^2=0.\]
		With the above analysis, 
		\[ \|z^j - z^{j+1} \|^2\leq \epsilon \]
		holds for sufficiently small $\epsilon>0$ with finite index $j$ for the upper level of Algorithm \ref{alg: Globalization of ALADIN} and \ref{alg: Globalization of ADMM}.
		Moreover, with the definition of merit function \eqref{eq: Merit}, the limit point of $\Phi^{(\cdot,\cdot)}(\cdot)$ is, at least, a \emph{critical point} of Problem \eqref{eq: DOPT_C}.
		To test whether the critical point is a local optima is trivial. By applying a perturbation $\varsigma\in \mathbb{R}^n \rightarrow 0$ to $z^*$, serving as a new initial point, Algorithm \ref{alg: Globalization of ALADIN} and \ref{alg: Globalization of ADMM} restart iterating until they converge to $z^*$ again, which means that $z^*$ is a \emph{local minimizer}. Otherwise,
		a new \emph{critical point} will be found, which indicates that $z^*$ is a saddle point. Note that the above global convergence analysis is also applicable to general ADMM and T-ALADIN  which is not shown here.
		
		Note that,	for convex cases, the proposed globalization strategy is useless. For non-convex cases,  
		the proposed globalization strategy can not guarantee to find a global minimizer of the  non-convex distributed problem \eqref{eq: DOPT_C} but can, at least, find a local optimizer.

		\begin{remark}
			In Algorithm \ref{alg: Globalization of ALADIN} and \ref{alg: Globalization of ADMM}, 
			due to the strongly convexity of Problem \eqref{eq: DOPT_C2}, there exists a parameter $\delta(z^k)>0$ (with fixed $z^k$) such that the lower global variable $y$ converges to the global optimizer of Problem \eqref{eq: DOPT_C2} with linear rate $\left(\frac{1}{\sqrt{1+\delta(z^k)}}\right)$ \cite{DLM,Du2023}.
			Assume that  there exists a $\bar \delta\geq \delta(z^k)\;\forall z^k$,  then, with our proposed bi-level globalization strategy,  the upper global variable $z$ will, at least, globally converge to a local optimizer of Problem \eqref{eq: DOPT_C} with linear convergence rate $\left(\frac{1}{\sqrt{1+\bar \delta}}\right)$.
		\end{remark}
		
		\begin{remark}
			If we skip the local update of Algorithm \ref{alg: Globalization of ALADIN} and update the lower global variable $y$ in a decentralized fashion, then it relates to another algorithm, called a consensus version of \emph{decentralized bi-level SQP} \cite{stomberg2022decentralized}. In this mentioned bi-level structure, our proposed globalization strategy is redundant for convex cases. However, the proposed globalization strategy still helps to establish the global convergence theory for non-convex cases. 
		\end{remark}
		\section{Conclusion}\label{sec: conclusion}
		In this paper, a novel bi-level globalization strategy of C-ADMM and C-ALADIN for non-convex problems is proposed.
		Without the help of the Lyapunov function or the corresponding Lagrangian function, we establish the global convergence theory for non-convex problems based on the state-of-the-art convex analysis of C-ADMM and C-ALADIN.
		Future work will concentrate on 
		non-convex distributed consensus optimization with constrained sub-problems. Decentralized variation of the proposed globalization strategy will also be considered.

		\section{Acknowledgments}
		We are grateful to Boris Houska, Bingsheng He, Shijie Zhu,  Alexander Engelmann, Kai Wang, Yuning Jiang and Xiaojun Yuan for their helpful discussion.
		\appendices

		\section{Comments on Local Variable Update in Section~\ref{eq: Conservative Operation} }
		\label{app: cmt}
		
		Other approximation techniques  can be also implemented here. For C-ALADIN, a linearization of the lower objective $f_i(x_i)$ is shown as Equation \eqref{eq: local update0}.
		\begin{equation}\label{eq: local update0}
			\begin{split}
				{x_i}^+
				\approx& \mathop{\arg\min}_{x_i} \partial f_i(y)^\top(x_i-y)+     \frac{\gamma}{2} \|x_i  -  z\|^2   \\
				&   +\lambda_i^\top  x_i+\frac{1}{2}\|x_i-y\|^2_{B_i}\\
				=
				&	\left( B_i+\gamma I \right)^{-1}\left(\gamma z+B_i y - \lambda_i -\partial f_i(y)  \right).
			\end{split}
		\end{equation}
		The local primal variable update can be also designed as a multiple iteration version, as shown in Equation \eqref{eq: multi1} where $x_i$ can be updated within finite iterations
		\begin{equation}\label{eq: multi1}
			\left\{
			\begin{split}
				x_i^- &= x_i^+ \quad\text{(take $x_i^-=y$ as an initial)};\\
				x_i^+ &= \left(B_i+\gamma I \right)^{-1}\left( \gamma z +B_i y-\lambda_{i} -\partial f_i(x_i^-) \right).
			\end{split}\right.
		\end{equation}
		
		For C-ADMM,
		by setting $B_i = \rho I$ and with $\rho >0$, Equation \eqref{eq: local update0} boils down to \eqref{eq: local update2} and has a similar expression as \eqref{eq: multi2}
		\begin{equation}\label{eq: local update2}
			\begin{split}
				{x_i}^+=		\frac{1}{\gamma+\rho}\left(\gamma z+\rho y - \lambda_i -\partial f_i(y)  \right).
			\end{split}
		\end{equation}
		Similar as \eqref{eq: multi1},  \eqref{eq: local update2} can be also replaced by \eqref{eq: multi2}
		\begin{equation}\label{eq: multi2}
			\left\{
			\begin{split}
				x_i^- &= x_i^+
				\quad\text{(take $x_i^-=y$ as an initial)};\\
				x_i^+ &= \frac{1}{\gamma+\rho} \left(\gamma z+\rho y -\lambda_{i}-\partial f_i(x_i^-) \right)
			\end{split}\right.
		\end{equation}
		where $x_i$ is updated with finite iterations.
		A similar idea has been proposed in \cite{DLM} and shows a linear convergence rate of strongly convex ADMM. Moreover, a quadratic approximation method can be also considered \cite{mokhtari2016dqm}.
		
		\section{Proof of Lemma \ref{Lemma: directional derivative}}\label{app: Proof of Lemma1}
		The following prove inherits from \cite[Theorem 18.2]{Nocedal2006}.
		Assume that the merit function with a increment has the following linear representation
		\begin{equation}\label{eq: linear}
			\begin{split}
				&\Phi^{( z, y^+)}(x^++t\Delta\tilde x)\\
				= & \sum_{i=1}^{N} F_i^{ z}(x_i^++t\Delta\tilde x_i)+ \sum_{i=1}^{N}\sigma_i\| x_i^+-y^+ +t\Delta\tilde x_i \|_1\\
				=& \sum_{i=1}^{N} F^{ z}_i(x_i^+)+t \sum_{i=1}^{N} \nabla F^{ z}_i(x_i^+)^\top \Delta\tilde x_i\\
				&+ \sum_{i=1}^{N}\sigma_i\| x_i^+ -y^+ -t(x_i^+ -y^+)\|_1   +O(t^2).
			\end{split}
		\end{equation}
		By isolating the first order block with respect to $t$, Equation \eqref{eq: linear} can be reformulated  as follows
		\begin{equation}
			\begin{split}
				&\Phi^{( z, y^+)}(x^++t\Delta\tilde x)\\
				=	& t \left\{ \sum_{i=1}^{N} \nabla F^{ z}_i(x_i^+)^\top \Delta\tilde x_i- \sum_{i=1}^{N}\sigma_i\| \Delta\tilde x_i \|_1  
				\right\}\\
				&+\underbrace{\left\{\sum_{i=1}^{N} F^{ z}_i(x_i^+)   + \sum_{i=1}^{N}\sigma_i\| x_i^+ -y^+ \|_1 \right\}}_{	\Phi^{(z, y^+)}(x^+)}
				+O(t^2).
				\\
			\end{split}
		\end{equation}
		
		Then, from the definition of directional derivative \eqref{eq: directional derivative1}, we have the resultant derivation as
		\begin{equation}\label{eq: directional derivative}
			\begin{split}
				&\text{D} \Phi^{( z, y^+)}(x^+)[\Delta\tilde x] \\
				=&  \lim_{t\rightarrow 0 ^+}\frac{	\Phi^{( z, y^+)}(x^++t\Delta\tilde x)-\Phi^{( z, y^+)}(x^+)}{t}\\
				=&\sum_{i=1}^{N} \nabla F^{ z}_i(x_i^+)^\top \Delta\tilde x_i- \sum_{i=1}^{N}\sigma_i\| \Delta\tilde x_i\|_1 +O(t) .
			\end{split}
		\end{equation}
		This completes the proof.

		\section{Proof of Theorem \ref{theorem: decrease1}}\label{app: theorem1}

		First, we have a first order Taylor expansion of Equation \eqref{eq: local update},
		\begin{equation}\label{eq: NLP approx}
			\begin{split}
				&\mathop{\min}_{x_i} F^{ z}_i(x_i)+\lambda_i^\top  x_i+\frac{1}{2}\|x_i-y\|^2_{B_i}\\
				\approx &  \nabla F^{ z}_i(y)^\top\Delta x_i +\lambda_i^\top(y+\Delta x_i)  +\frac{1}{2}\|\Delta x_i\|^2_{B_i}.
			\end{split}
		\end{equation}
		Clearly, the first order optimality condition of Equation \eqref{eq: NLP approx} shows as 
		\begin{equation}\label{eq: equality}
			\begin{split}
				& \nabla F^{ z}_i(y)  +\lambda_i  +B_i  \Delta x_i=0.
			\end{split}
		\end{equation}
		By multiplying $\Delta x_i$ to the left hand side of \eqref{eq: equality} and then performing summation over the agent index, we get the following equality
		\begin{equation}\label{eq: NLP inequality}
			\begin{split}
				&\sum_{i=1}^{N} \nabla F^{ z}_i(y)^\top \Delta x_i
				=  \sum_{i=1}^{N} \left(-\lambda_i^\top \Delta x_i- \|\Delta x_i\|^2_{B_i}\right).\\
			\end{split}
		\end{equation}
		With Equation \eqref{eq: directional derivative2} and \eqref{eq: NLP inequality},
		if  \eqref{eq: assumption} holds, then \eqref{eq: purpose} is guaranteed, which completes our proof.

		\section{Proof of Theorem \ref{theorem: decrease}}\label{app: Theorem2}
		
		From the first order optimal condition of Equation~\eqref{eq: consensus QP1}  with respect to $\Delta\tilde{ x}_i$, the follow equality always holds.
		\begin{equation}
			B_i \Delta\tilde{ x}_i + \lambda_i^+ + \nabla F^{ z}_i(x_i^+) = 0.
		\end{equation}
		We then multiply $\Delta\tilde{ x}_i$ to the left hand side and get 
		\begin{equation}
			\Delta\tilde{ x}_i ^\top B_i \Delta\tilde{ x}_i + (\lambda_i^+) ^\top \Delta\tilde{ x}_i +\nabla F^{ z}_i(x_i^+) ^\top \Delta\tilde{ x}_i =0.
		\end{equation}
		This implies 
		\begin{equation}\label{eq: global inequality}
			\begin{split}
				&\sum_{i=1}^{N} \nabla F^{ z}_i(x_i^+)^\top \Delta\tilde{ x}_i\\
				=& - \sum_{i=1}^{N}\left ( \Delta\tilde{ x}_i^\top B_i \Delta\tilde{ x}_i   +  (\lambda_i^+) ^\top \Delta\tilde{ x}_i \right)\\
				\leq&   \sum_{i=1}^{N}\|\lambda_i^+\|_\infty \| \Delta\tilde{x}_i  \|_1- \sum_{i=1}^{N} \|\Delta\tilde{ x}_i\|_{B_i}^2.
			\end{split}
		\end{equation}
		With Equation \eqref{eq: condition} and $B_i\succ0$, once we plug  \eqref{eq: global inequality} into \eqref{eq: directional derivative0},
		Equation \eqref{eq: consensus descent} holds. This completes the proof.


		\section{Proof of Theorem \ref{the: lya}}\label{app: lya}
		The following proof is similar to that of \cite[Appendix C]{Du2023}.

		We then show that, by proving a non-positive equation \eqref{eq: equivalent}, Equation \eqref{eq: lya} is monotonically decreasing. Specifically,
		\begin{equation}\label{eq: equivalent}
			4\sum_{i=1}^{N} \left( \lambda_{i} - \lambda_i^* +B_i (x_i^+-y) \right)^\top \left( x_i^+-y^* \right)
		\end{equation}
		is equivalent to 
		\[	\mathcal L(y^+,\lambda^+)-\mathcal L(y,\lambda).\]

		
		Equation \eqref{eq: equivalent} can also be expressed as 
		\begin{equation}\label{eq: middle}
			\begin{split}
				&4\sum_{i=1}^{N} \left( \lambda_{i} - \lambda_i^* +B_i (x_i^+-y) \right)^\top \left( x_i^+-y^* \right)\\
				=&4\sum_{i=1}^{N}\left( \lambda_{i} - \lambda_i^* \right)^\top \left( x_i^+-y^* \right)\\
				& +4\sum_{i=1}^{N}\left( B_i (x_i^+-y)\right)^\top \left( x_i^+-y^* \right).
			\end{split}
		\end{equation}
		Notice that from the KKT optimality condition of Equation \eqref{eq: DOPT_C2}, we have
		\begin{equation}\label{eq: lam}
			\sum_{i=1}^{N} \lambda_{i} =0.
		\end{equation}
		By applying \eqref{eq: lam},  
		Equation \eqref{eq: middle} becomes as
		\begin{equation}\label{eq: middle2}
			\begin{split}
				&4\sum_{i=1}^{N}\left( \lambda_{i} - \lambda_i^* \right)^\top \left( x_i^+-y \right)\\
				& +4\sum_{i=1}^{N}\left( B_i (x_i^+-y)\right)^\top \left( x_i^+-y^* \right)\\
				=& 4\sum_{i=1}^{N}\left( \lambda_{i} - \lambda_i^* +B_i\left( x_i^+-y^* \right) \right)^\top\left( x_i^+-y \right).
			\end{split}
		\end{equation}
		In the end,	by introducing an important formula appeared in \cite{Du2023} 
		\[ x_i^+ = \frac{B_i^{-1}}{2} \left( \lambda_{i}^+-\lambda_{i}\right) + \frac{1}{2} \left(  y+y^+ \right), \]
		we simplify Equation \eqref{eq: middle2} as
		\begin{equation}
			\begin{split}
				&4\sum_{i=1}^{N}\left( \lambda_{i} - \lambda_i^* +B_i\left( x_i^+-y^* \right) \right)^\top\left( x_i^+-y \right)\\
				=&  \sum_{i=1}^{N} \| \lambda_{i}^+-\lambda_{i}^*\|^2_{B_i^{-1}}- \sum_{i=1}^{N} \| \lambda_{i}-\lambda_{i}^*\|^2_{B_i^{-1}}\\
				&+ \sum_{i=1}^{N} \| y^+-y^*\|^2_{B_i}
				- \sum_{i=1}^{N} \| y-y^*\|^2_{B_i}\\
				=& 	\mathcal L(y^+,\lambda^+)-\mathcal L(y,\lambda).
			\end{split}
		\end{equation}
		As we mentioned, Equation \eqref{eq: equivalent} is non-positive, therefore
		\[ 	\mathcal L(y^+,\lambda^+)-\mathcal L(y,\lambda)\leq 0 \]
		is guaranteed. This completes our proof.
		
		\section{Proof of Theorem \ref{theorem: global convergence}} \label{APP: global convergence}
		In the lower level, Problem \eqref{eq: DOPT_C2} of  both algorithms is strongly convex. Algorithm~\ref{alg: Globalization of ALADIN} and \ref{alg: Globalization of ADMM} will guarantee that, with finite iteration of the lower level, a better minimizer $y^{+}$ can be found such that
		\begin{equation}\label{eq: consensus QP}
			\begin{split}
				&	\sum_{i=1}^{N} f_i( z) \\=&\underbrace{ \sum_{i=1}^{N} \left(f_i( z) + \frac{\gamma}{2} \| z- z\|^2\right)+\sum_{i=1}^{N}\sigma_i \| z -z  \|_1}_{\overset{\eqref{eq: Merit}}{=}\Phi^{( z,{z})}(z)} \\
				\overset{\eqref{eq: lower}}{>}
				&\underbrace{\sum_{i=1}^{N}\left( f_i(y^+) +\frac{\gamma}{2} \|  y^+- z \|^2\right)+\sum_{i=1}^{N}\sigma_i \| y^+ -y^+  \|_1}_{\overset{\eqref{eq: Merit}}{=}\Phi^{(z,y^+)}(y^+)}.
			\end{split}
		\end{equation}
		By adding $\frac{\gamma N}{2} \|  y^+- y^+ \|^2$ on the right hand side of \eqref{eq: consensus QP}, we can easily show
		\begin{equation*}
			\Phi^{({z},{z})}( z) -\Phi^{(y^+,y^+)}(y^+) >   \frac{\gamma N}{2} \|  y^+- z \|^2.
		\end{equation*}
		This completes the proof.

		\bibliographystyle{unsrt}
		\bibliography{paper}
		\balance

	\end{document}